\documentclass[11pt,english]{article}
\usepackage[margin=3 cm,bottom=35mm,footskip=15mm,top=35mm]{geometry}
\usepackage{enumerate,url,amssymb,amsthm,amsmath,boolexpr, xstring,enumitem,hyperref,subcaption,floatrow,setspace}
\usepackage[numbers]{natbib}
\usepackage[nameinlink,capitalise,noabbrev]{cleveref}
\crefname{appsec}{Appendix}{Appendices}
\crefformat{equation}{#2(#1)#3}

\theoremstyle{plain}

\newtheorem{theorem}{Theorem}[section]

\newtheorem{lemma}[theorem]{Lemma}
\newtheorem{claim}[theorem]{Claim}
\newtheorem{corollary}[theorem]{Corollary}

\newtheorem*{question*}{Question} \Crefname{question}{Question}{Questions}

\theoremstyle{definition}
\newtheorem{definition}[theorem]{Definition}
\newtheorem{question}[theorem]{Question}

\theoremstyle{remark}

\newcommand{\F}{\mathcal{F}}

\newcommand{\DVC}{\mathrm{VC_{dual}}}

\newcommand{\xqed}[1]{%
	\leavevmode\unskip\penalty9999 \hbox{}\nobreak\hfill
	\quad\hbox{\ensuremath{#1}}}
\newcommand{\Endofdef}{\xqed{\lozenge}}

\title{The extremal number of Venn diagrams}

\author{
	Peter Keevash\thanks{Mathematical Institute, University of Oxford, Oxford, UK. Email:
		\href{mailto:keevash@maths.ox.ac.uk} {\nolinkurl {keevash@maths.ox.ac.uk}}. Research supported in part by ERC Consolidator Grant 647678.}
	\and
	Imre Leader\thanks{Centre for Mathematical Sciences, University of Cambridge, Cambridge, UK. Email:
		\href{mailto:i.leader@dpmms.cam.ac.uk} {\nolinkurl{i.leader@dpmms.cam.ac.uk}}.
	}
	\and
	Jason Long\thanks{Mathematical Institute, University of Oxford, Oxford, UK. Email:
		\href{mailto:jason.long@maths.ox.ac.uk} {\nolinkurl{jason.long@maths.ox.ac.uk}}.
	}
	\and
	Adam Zsolt Wagner\thanks{Department of Mathematics, ETH, Z\"urich, Switzerland. Email:
		\href{mailto:zsolt.wagner@math.ethz.ch} {\nolinkurl{zsolt.wagner@math.ethz.ch}}.}
}

\begin{document}
	
	\maketitle
	
	\begin{abstract}
		We show that there exists an absolute constant $C>0$ such that any family $\mathcal{F}\subset \{0,1\}^n$ of size at least $Cn^3$ has dual VC-dimension at least 3. Equivalently, every family of size at least $Cn^3$ contains three sets such that all eight regions of their Venn diagram are non-empty. This improves upon the $Cn^{3.75}$ bound of Gupta, Lee and Li and is sharp up to the value of the constant.
	\end{abstract}
	
	\section{Introduction}
	We study an extremal problem concerning the maximum size of a set system avoiding a certain forbidden configuration. Such problems are ubiquitous in combinatorics, statistics and theoretical computer science, and are the focus of a number of fundamental results and conjectures. One central notion is that of the \emph{Vapnik-Chervonenkis} dimension which plays an important role in statistical learning theory~\cite{Vapnik} and discrete and computational geometry (see~\cite{VCgeo} and the surveys~\cite{VCgeosurvey1},~\cite{VCgeosurvey2}). 
	
	We write $\mathcal{P}(n)$ for the powerset of $[n]$. A family $\F \subset \mathcal{P}(n)$ \emph{shatters} a set $S\subset [n]$ if for all $A\subset S$ there exists a set $B\in  \F$ with $B\cap S=A$.  The \emph{Vapnik-Chervonenkis} dimension, or \emph{VC-dimension} for short, of a family $\F \subset \mathcal{P}(n)$ is defined as $$\mathrm{VC}(\F) = \max\{|S|:\F \text{ shatters } S\}.$$
	
	A cornerstone result of extremal combinatorics due to Sauer and Shelah bounds the size of a family in terms of its VC-dimension.
	\begin{lemma}[Sauer--Shelah~\cite{sauer}]\label{lem:sauershelah}
		For any family $\F\subset \mathcal{P}(n)$ we have $$|\F|\leq\sum_{k=0}^{\mathrm{VC}(\F)}\binom{n}{k}.$$
	\end{lemma}
	
	In the present paper we consider the dual notion of VC-dimension. Given a family $\F\subset \mathcal{P}(n)$, we can view $\F$ as a $0/1$ incidence matrix $F$ of dimension $n\times |\F|$ with rows indexed by $x\in[n]$ and columns indexed by $A\in\F$. The \emph{dual VC-dimension} of $\F$, written as $\DVC(\F)$, is then simply the VC-dimension of $F^T$, the transpose of the matrix $F$. Equivalently, the dual VC-dimension of a family $\F$ is the largest $k\in\mathbb{N}$ such that there exist sets $A_1,A_2,\ldots,A_k\in\F $ with all $2^k$ regions of the form $B_1\cap B_2\cap \ldots \cap B_k$ where $B_i\in\{A_i,[n]\setminus A_i\}$ being non-empty. We say that such sets $A_1,\dots,A_k$ form a \emph{$k$-Venn diagram}.
	
	The existence of a $k$-Venn diagram in $\F $ corresponds to the presence of a certain submatrix of the matrix $F$ -- specifically, a row-column permutation of the matrix $M_k$ with $k$ columns and $2^k$ rows given by all possible binary sequences of length $k$. This interpretation places the dual VC-dimension in the context of widely studied problems on \emph{forbidden configurations}~\cite{ansteesurvey}. A highly influential conjecture is that of Anstee and Sali~\cite{anstee}, which predicts (up to a constant factor) the maximum number of edges in a hypergraph $\F$ before the corresponding matrix $F$ contains a row-column permutation of a certain submatrix $M$. This prediction involves a quantity $X(M)$ which is NP-hard to calculate~\cite{ansteeNPhard} in general, but straightforward to calculate for small $M$. This conjecture has been verified for cases in which the number of rows is small~\cite{ansteespecialcases1,anstee}, but most cases remain open when the number of rows is at least six.
	
	A natural class of special cases of the conjecture of Anstee and Sali is obtained by considering the matrices $M_k$ described above. Understanding the corresponding forbidden configuration problem would provide a dual version of the Sauer-Shelah lemma, giving the maximum possible size of a family $\F$ on ground set $[n]$ with $\DVC(\F)\le k$.
	
	For $k=1$, such a result is well known. We say that two sets $A,B\subset [n]$ form a  \emph{crossing pair} if all four sets $A\cap B, \overline{A}\cap B, A\cap \overline{B}$ and $\overline{A}\cap \overline{B}$ are non-empty. A family $\F\subseteq\mathcal{P}(n)$ has $\DVC(\F)\le 1$ if and only if $\F$ does not contain sets $A_1$ and $A_2$ such that $A_1$ and $A_2$ are a crossing pair, and the maximum size of such a family is $4n-2$, see~\cite{edmonds}. More recently, the problem of bounding the size of a family avoiding $k$ pairwise crossing sets $A_1,\dots,A_k$ has been considered. Denoting by $g_k(n)$ the maximum size of a family on ground set $[n]$ that does not contain $k$ sets that are pairwise crossing, Karzanov and Lomonosov~\cite{karzanov} conjectured that $g_k(n)=O_k(n)$. The best upper bound is $g_k(n)=O_k(n\log^*n)$, as shown by    Kupavskii, Pach and Tomon~\cite{pityu}.
	
	The next step towards a dual Sauer-Shelah lemma is therefore to bound the size of a family $\F$ for which $\DVC(\F)\le 2$. To achieve this, we wish to bound the size of a family  $\F\subset\mathcal{P}(n)$ given that $\F$ is not allowed to contain a $3$-Venn diagram, that is to say three sets $A,B,C\in\F$ such that all eight regions $A\cap B\cap C, \overline{A}\cap B \cap C, \ldots, \overline{A}\cap\overline{B}\cap\overline{C}$ are non-empty. A lower bound of the form $cn^3$ follows by considering the family of sets of size at most three, while an upper bound of the form $Cn^7$ follows easily from the Sauer-Shelah lemma. The Anstee-Sali conjecture predicts that the lower bound is correct up to the value of $c$.
	
	A connection between dual VC-dimension and the performance of an algorithm of Karger and Stein~\cite{KargerStein} for finding minimal $k$-cuts was noted by Gupta, Lee and Li~\cite{Gupta}. For an edge-weighted graph $G=(V,E)$, a minimal $k$-cut is a subset $E'\subset E$ of minimal weight such that $G'=(V,E\setminus E')$ has at least $k$ connected components. By showing that if $\F$ is a family of subsets of $[n]$ with $\DVC(\F)\leq 2$ then $|\F|\leq Cn^{3.75}$, they were able to give an algorithm enumerating the minimal $k$-cuts of an $n$ vertex graph in time $\mathcal{O}(n^{(2-\epsilon)k})$ for an explicit $\epsilon>0$, improving on the previous best known bounds of $n^{(2-o(1))k}$.
	
	Our main result is closing the remaining polynomial gap for the maximum size of a set family of dual VC-dimension at most two.
	
	\begin{theorem}\label{thm:main}
		If $\F$ is a family of subsets of $[n]$ and $\DVC(\F)\leq 2$ then $|\F|= O(n^3)$.
	\end{theorem}
	
	This can be seen as a further step towards a dual Sauer-Shelah lemma and as a resolution of a natural case of the Anstee-Sali conjecture. Furthermore, Theorem~\ref{thm:main} can be directly applied to give a small improvement to the factor $2-\epsilon$ in the exponent of the theoretical running time of the algorithm described in~\cite{Gupta} for enumerating minimal $k$-cuts, although the calculations required to determine the new value of $\epsilon$ are involved so we omit them here.
	
	\section{Setting up}

	For any set $A\subset [n]$ let $\overline{A}:=[n]\setminus A$. Given three sets $A,B,C$, let $\mathrm{V}_3(A,B,C)$ be the eight-element multiset $\{A\cap B\cap C, \overline{A}\cap B \cap C, \ldots, \overline{A}\cap \overline{B}\cap \overline{C}\}$, which we will refer to as the 3-Venn diagram of $A,B,C$. Recall that we will be primarily interested in the number of non-empty regions in such Venn diagrams. We refer to the set $A\cap B \cap C$ as the \emph{innermost region} and to $\overline{A}\cap \overline{B}\cap \overline{C}$ as the \emph{outermost region} of their Venn diagram. We begin by recalling a result from~\cite{Gupta}.
	\begin{lemma}\label{lem:gupta}[Lemma 5.11 of~\cite{Gupta}]
		If $\mathcal{F}$ is a family on ground set $[n]$ of size at least $8n$ then there are sets $A,B,C$ so that at least four out of the six sets in $\mathrm{V}_3(A,B,C)\setminus \{A\cap B \cap C, \overline{A}\cap \overline{B}\cap \overline{C}\}$ are non-empty.
	\end{lemma}

	Given three sets $A,B,C$ we refer to some proper subset $\mathrm{V}'(A,B,C)$ of the set $\mathrm{V}_3(A,B,C)$ of regions as a \emph{partial 3-Venn diagram}. The lemma above shows that $\mathcal{O}(n)$ sets are enough to find three sets $A,B,C$ such that the partial 3-Venn diagram $\mathrm{V}'(A,B,C)=\mathrm{V}_3(A,B,C)\setminus \{A\cap B \cap C, \overline{A}\cap \overline{B}\cap \overline{C}\}$ has at least four non-empty regions. The following lemma, whose proof is very similar to the inductive proof of the Sauer-Shelah lemma and appears as Lemma~5.9 in~\cite{Gupta}, allows us to increase the number of filled regions at the cost of a factor of $n$.
	
	\begin{lemma}\label{lem:boosting}
		Let $\mathrm{V}'(A,B,C)$ be a partial 3-Venn diagram with $r$ regions. Suppose that there exists a constant $C$ such that in any family $\mathcal{F}\subset\mathcal{P}(n)$ of size at least $Cn^k$ we can find sets $A,B,C$ in $\F$ such that $\mathrm{V}'(A,B,C)$ has $s<r$ regions filled. Then there exists a constant $C'$ such that in any family $\mathcal{F}'\subset\mathcal{P}(n)$ of size at least $C'n^{k+1}$ we can find sets $A,B,C$ in $\F'$ such that $\mathrm{V}'(A,B,C)$ has at least $s+1$ regions filled.
	\end{lemma}
	\begin{proof}
		We prove the result by induction on $n$ -- the base case is trivial for $C'\geq C$ sufficiently large.
		
		For the inductive step, we let $\mathcal{F}'\subset\mathcal{P}(n)$ have size at least $C'n^{k+1}$ and consider the families $\mathcal{F}_1'=\{F\in \mathcal{F}': 1\in F\}$ and $\mathcal{F}_2'=\{F\in \mathcal{F}': 1\not\in F\}$. Then $|\mathcal{F}'|=|\mathcal{F}_1'\cup\F_2'|+|\F_1'\cap\F_2'|$ and both $\F_1$ and $\F_2$ can be regarded as families on an $(n-1)$-element ground set. As $|\F'|\ge C'n^{k+1}$ we must either have $|\F_1\cup\F_2|\ge C'(n-1)^{k+1}$ in which case we are done by induction, or $|\F_1\cap\F_2|\ge C'n^k$. In the latter case, we can find sets $A',B',C'$ such  that $\mathrm{V}'(A',B',C')$ has $s$ regions filled. But $\F'$ contains all the sets $\{A',B',C',A'\cup\{1\},B'\cup\{1\},C'\cup\{1\}\}$ and so we can choose $A\in\{A',A'\cup\{1\}\}$, $B\in\{B',B'\cup\{1\}\}$ and $C\in\{C',C'\cup\{1\}\}$ such that the element $1$ belongs to an empty region of $\mathrm{V}'(A',B',C')$. Thus $\mathrm{V}'(A,B,C)$ has at least $s+1$ regions filled.
	\end{proof}
	Combining Lemmas~\ref{lem:gupta} and \ref{lem:boosting}, we have the following corollary.
	\begin{corollary}\label{cor:cor1}
		There exists a constant $D$ such that if $\mathcal{F}$ is a family on ground set $[n]$ of size at least $Dn^4/k$ such that all members of $\mathcal{F}$ have size between $k$ and $n-k$ then $\mathcal{F}$ contains a 3-Venn with all eight regions being non-empty.
	\end{corollary}
	\begin{proof}
		From~\cref{lem:gupta} and~\cref{lem:boosting} we know that there exists a constant $C$ such that if a family in $\mathcal{P}(n)$ has size at least $Cn^3$ then it contains a 3-Venn diagram with all regions, except possibly the innermost and outermost regions, non-empty. Suppose now we are given a family $\F\subset \mathcal{P}(n)$ of size at least $2Cn^4/k$. Without loss of generality we may assume that at least half of the sets in $\F$ do not contain the element 1 -- if that is not the case, we may replace every set in $\F$ by their complement. Let $\F':=\{F\in\F:1\not\in F\}$ so that $|\F'|\geq |\F|/2\geq Cn^4/k$. As in $\F'$ every set has size at least $k$, there is an element $x\in [n]$ such that $|\F'_{x}|=|\{F\in\F':x\in F\}|\geq Cn^3$. Hence $\F'_x$ contains three sets that form a 3-Venn diagram with all regions non-empty, except possibly the outermost and innermost regions. However, together with the elements $1$ and $x$ we find that all eight regions must be non-empty. Setting $D=2C$ finishes the proof.
	\end{proof}
	
	We will prove the following lemma, from which~\cref{thm:main} follows easily.
	
	\begin{lemma}\label{lem:main}
		There exists $C$ such that if $\mathcal{F}$ is a family on the ground set $[n]$ of size at least $Cn^{2}$ then $\mathcal{F}$ contains a 3-Venn with six of the inner seven regions filled.
	\end{lemma}
	
	This lemma will follow from a structural lemma, whose precise statement will require some set-up. Essentially, the idea behind the proof of~\cref{thm:main} is that if we could prove that any family $\mathcal{F}$ of size larger than $Cn$ contained a 3-Venn with six non-empty regions then~\cref{thm:main} would follow easily from the proof method of the Sauer-Shelah Lemma(\cref{lem:sauershelah}). However, there are examples of families with super-linear size that do not contain six regions of a 3-Venn: consider, for example, the family of sets of size two. Our aim will be to show that in fact all super-linear families avoiding six regions of a 3-Venn have some structure related to that of the family of pairs. We will then use this structural information to provide an improvement over the usual Sauer-Shelah induction step in this case.

	We now provide some definitions that allow us to describe the structure that we hope to find.
	
	\begin{definition}\label{treeg} 
		A family $\mathcal{F}$ is \emph{pair-like with respect to a family $\mathcal{P}$} if the sets in $\mathcal{F}$ are all equal to $P_u\cup P_v$ for disjoint $P_u$ and $P_v$ in $\mathcal{P}$, with the additional property that for each  pair of disjoint $P_u,P_v\in\mathcal{P}$ with $P_u\cup P_v\in \mathcal{F}$ there exists
		\begin{itemize}
			\item a family $\mathcal{P}_{uv}$ consisting of at least 10 different $P_w$ such that $P_u\cup P_w\in\mathcal{F}$ for each $P_w\in \mathcal{P}_{uv}$ and the family $\mathcal{P}_{uv}\cup\{P_u,P_v\}$ consists of disjoint sets,
			\item and similarly a family $\mathcal{P}_{vu}$ consisting of at least 10 different $P_w$ such that $P_v\cup P_w\in\mathcal{F}$ for each $P_w\in \mathcal{P}_{vu}$ and the family $\mathcal{P}_{vu}\cup\{P_u,P_v\}$ consists of disjoint sets.
		\end{itemize}
		We call $\mathcal{P}$ the \emph{basis} of $\mathcal{F}$. If $F\in \mathcal{F}$ is written $P_u\cup P_v$, we call $P_u$ and $P_v$ the \emph{components} of $F$. We may simply say that $\mathcal{F}$ is pair-like if there exists some suitable basis $\mathcal{P}$.
	\end{definition}
	
	Given a family $\mathcal{F}$ which is pair-like with respect to some family $\mathcal{P}$, we say that some set $P_u\in\mathcal{P}$ is \emph{popular} if $P_u$ is the component of some set $F\in\mathcal{F}$ (and thus the component of at least 10 such sets).
	
	In the next section we will state and prove the structural lemma discussed earlier, and then in~\cref{sec:deduction} we will show how this lemma may be used to prove~\cref{lem:main}.
	
	\section{The structural lemma}
	
	For technical reasons, it is convenient to treat the outermost region somewhat differently to the others. Thus our results in this section will refer to the inner seven regions of the 3-Venn, meaning all regions except $\overline{A}\cap\overline{B}\cap\overline{C}$.
	
	\begin{lemma}\label{lem:struct}
		There exist positive constants $\alpha, \beta$ such that the following holds. Let $\mathcal{F}$ be a family on the ground set $[n]$ that avoids a 3-Venn diagram with five of the inner seven regions filled. Then $\mathcal{F}=\mathcal{F}_1\cup\mathcal{F}_2$ where $|\mathcal{F}_1|\le \alpha n-\beta$ and $\mathcal{F}_2$ is pair-like.
	\end{lemma}
	
	\begin{proof}
		This lemma will be proved by induction. The main idea will be to find collections of elements such that only a relatively small number of sets distinguish some two of those elements (meaning that the set contains some, but not all, of the elements). We will then collapse this collection of points into a single point and use induction, while adding the sets that distinguished those points to a `junk pile' which will become $\mathcal{F}_1$.
		
		Let $\mathcal{F}$ be a family that does not contain a 3-Venn diagram with five of the inner seven regions filled. If $|\F|<4n$ there is nothing to prove -- otherwise as remarked in the introduction, we have $\DVC(\F)>1$ and hence there exists a crossing pair in $\F$, i.e.~two sets $F_1,F_2\in\F$ such that all four sets $A:=F_1\cap F_2, B=F_1\cap \overline{F_2}, C = \overline{F_1}\cap F_2$ and $D=\overline{F_1}\cap \overline{F_2}$ are non-empty. Note that since the family $\mathcal{F}$ avoids six regions of the Venn diagram, any third set $F_3$ in $\mathcal{F}$ can properly split at most one of the regions $A,B,C,D$. (We say that a set $S$ properly splits $T$ if both $S\cap T$ and $T\setminus S$ are non-empty.)
		
		Let $\mathcal{F}_A$, $\mathcal{F}_B$, $\mathcal{F}_C$ and $\mathcal{F}_D$ be the disjoint subfamilies of $\mathcal{F}$ that properly split $A,B,C$ and $D$ respectively. Note that there are at most 16 sets in $\F$ which do not properly split any of $A$, $B$, $C$ or $D$ since such a set must be the union of some subset of $\{A,B,C,D\}$. Therefore $$|\mathcal{F}|\le |\mathcal{F}_A|+|\mathcal{F}_B|+|\mathcal{F}_C|+|\mathcal{F}_D|+16.$$
		
		We pick arbitrary representative points $a\in A, b\in B, c\in C$ and $d\in D$. We let $A^+=A\cup\{b,c,d\}$, $B^+=B\cup\{a,c,d\}$, $C^+=C\cup\{a,b,d\}$ and $D^+=D\cup\{a,b,c\}$. Let the projection of the family $\mathcal{F}_A$ onto $A^+$ be denoted $\mathcal{F}_A^+$ and similarly for $B,C,D$. Note that $|\mathcal{F}_A|=|\mathcal{F}_A^+|$, as sets in $\F_A$ do not differ outside of $A$.
		
		Assume that at least two of $|A|,|B|,|C|$ and $|D|$ are bigger than 1. In this case, we can apply the induction hypothesis to each of the families $\mathcal{F}_A^+,\dots,\mathcal{F}_D^+$ since the ground sets have size strictly smaller than $n$.
		
		This allows us to partition $\mathcal{F}_A^+$ into $\mathcal{G}_A^+$ and $\mathcal{H}_A^+$, where $|\mathcal{G}_A^+|\le \alpha(|A|+3)-\beta$ and $\mathcal{H}_A^+$ is pair-like with respect to some basis $\mathcal{P}_A^+\subset \mathcal{P}(A^+)$. We get corresponding statements with $A$ replaced with $B,C$ and $D$.
		
		By replacing the points $b$, $c$ and $d$ in $\mathcal{P}_A^+$ by the sets $B$, $C$ and $D$ (meaning e.g. that a set $\{b,c,x\}$ would become $B\cup C\cup\{x\}$) we get a family $\mathcal{P}_A$. We obtain $\mathcal{P}_B$, $\mathcal{P}_C$ and $\mathcal{P}_D$ similarly. 
		
		Then by replacing the points $b$, $c$ and $d$ in $\mathcal{H}_A^+$ by the sets $B$, $C$ and $D$ we obtain a subfamily of $\mathcal{F}$ which is pair-like with respect to $\mathcal{P}_A$, which we call $\mathcal{H}_A$. We obtain $\mathcal{H}_B$, $\mathcal{H}_C$ and $\mathcal{H}_D$ similarly.
		
		Note that as $\F_A,\F_B,\F_C,\F_D$ were disjoint, so are $\mathcal{H}_A,\mathcal{H}_B,\mathcal{H}_C,\mathcal{H}_D$. We now claim that the family $\mathcal{H}=\mathcal{H}_A\cup\mathcal{H}_B\cup\mathcal{H}_C\cup\mathcal{H}_D$ is pair-like with respect to the basis $\mathcal{P}=\mathcal{P}_A\cup\mathcal{P}_B\cup\mathcal{P}_C\cup\mathcal{P}_D$.
		
		It is clear that the sets in $\mathcal{H}$ are indeed all equal to $P_u\cup P_v$ for disjoint $P_u,P_v\in\mathcal{P}$. Moreover, for any disjoint $P_u,P_v\in\mathcal{P}$ with $P_u\cup P_v\in\mathcal{H}$, there must be some choice of $X\in\{A,B,C,D\}$ such that $P_u\cup P_v\in \mathcal{H}_X$. Then the existence of the desired family $\mathcal{P}_{uv}$ follows from the fact that $\mathcal{H}_X$ is pair-like with respect to $\mathcal{P}_X\subset \mathcal{P}$.
		
		Recall that the size of $\mathcal{G}_A^+$ is at most $\alpha(|A|+3)-\beta$ and we have similar bounds for $\mathcal{G}_B^+$, $\mathcal{G}_C^+$ and $\mathcal{G}_D^+$. Therefore we have
		\begin{equation*} 
		\begin{split}
		|\mathcal{F}|&\le |\mathcal{F}_A|+|\mathcal{F}_B|+|\mathcal{F}_C|+|\mathcal{F}_D|+16\\
		&= |\mathcal{F}_A^+|+|\mathcal{F}_B^+|+|\mathcal{F}_C^+|+|\mathcal{F}_D^+|+16\\
		&=|\mathcal{G}_A^+|+|\mathcal{G}_B^+|+|\mathcal{G}_C^+|+|\mathcal{G}_D^+|+16+|\mathcal{H}|\\
		&\le |\mathcal{H}|+16+\alpha(|A|+3)+\alpha(|B|+3)+\alpha(|C|+3)+\alpha(|D|+3)-4\beta\\
		&\le|\mathcal{H}|+\alpha n-\beta
		\end{split}
		\end{equation*}
		for $\beta$ sufficiently large in terms of $\alpha$.
		
		This means that we have partitioned $\mathcal{F}$ into a pair-like family $\mathcal{F}_2=\mathcal{H}$ and a set $\mathcal{F}\setminus\mathcal{F}_2=\mathcal{F}_1$ with size at most $\alpha n-\beta$, and we are done in this case.
		
		It remains to consider what happens if every crossing pair in $\mathcal{F}$ has three regions of size exactly one, and one region of size $n-3$.
		
		In this case, we split $\mathcal{F}$ into $\mathcal{F}_{\text{small}}=\{F\in\mathcal{F}:|F|\le 2n/3\}$ and $\mathcal{F}_{\text{large}}=\{F\in\mathcal{F}:|F|>2n/3\}$. Note that any three sets in $\mathcal{F}_{\text{large}}$ have non-empty 3-wise intersection. If $|\mathcal{F}_{\text{large}}|>8n$ then by~\cref{lem:gupta} we can find three sets $A,B,C\in\F_{\text{large}}$ so that at least four out of the six sets in $\mathrm{V}_3(A,B,C)\setminus \{A\cap B \cap C, \overline{A}\cap \overline{B}\cap \overline{C}\}$ are non-empty. Since the innermost region is non-empty by the above, this gives a Venn diagram with five of the inner seven regions non-empty.
		
		Any crossing pair in $\mathcal{F}_{\text{small}}$ must in fact have the regions $A$, $B$ and $C$ each of size one (since the sets involved are too small for any of these regions to have size $n-3$). Therefore we may further split $\mathcal{F}_{\text{small}}$ into $\mathcal{F}_{\text{pairs}}=\{F\in \mathcal{F}_{\text{small}}:|F|=2\}$ and $\mathcal{F}_{\text{rest}}=\mathcal{F}_{\text{small}}\setminus \mathcal{F}_{\text{pairs}}$. Since $\mathcal{F}_{\text{rest}}$ cannot have any crossing pairs, $|\mathcal{F}_{\text{rest}}|\le 4n-2$. 
		
		Now we consider $\mathcal{F}_{\text{pairs}}$. This family corresponds to a set of pairs (edges) on the ground-set $[n]$ of size at least $\alpha n-\beta-12n+2$ (else we could have taken $\mathcal{F}_1=\mathcal{F}$ and $\mathcal{F}_2=\emptyset$). Provided that $\alpha n-\beta -12n+2>22n$ we can discard at most $11n$ pairs from $\F_{\text{pairs}}$ to obtain a subfamily $\F_2$ with minimum degree 11 (meaning that for any $x$ which appears as a member of some $X\in\F_2$ appears as a member of at least 11 distinct sets in $\F_2$). The subfamily $\mathcal{F}_2$ of $\mathcal{F}$ is therefore pair-like with respect to the basis of singletons. Note that $|\mathcal{F}\setminus\mathcal{F}_2|\le 12n-2+10n<\alpha n-\beta $ for $\alpha $ sufficiently large. This gives us the decomposition that we require in this case.
	\end{proof}
	
	\section{Deducing Lemma~\ref{lem:main} and Theorem~\ref{thm:main}}\label{sec:deduction}
	
	We begin with a useful lemma.
	
	\begin{lemma}\label{lem:useful}
		Let $\mathcal{A},\mathcal{B}$ and $\mathcal{C}$ be families on a ground set $[n]$ such that each family consists of at least seven pairwise disjoint and non-empty sets. Then we can find $A\in\mathcal{A}, B\in \mathcal{B}$ and $C\in\mathcal{C}$ such that $A\setminus(B\cup C), B\setminus(A\cup C)$ and $C\setminus (A\cup B)$ are each non-empty.
	\end{lemma}
	\begin{proof}
		We call a pair of sets $X$ and $Y$ \emph{weakly separated} if $X\setminus Y$ and $Y\setminus X$ are non-empty. First we claim that given families $U=\{U_1, U_2\}$ and $V=\{V_1,V_2\}$ of disjoint subsets of $[n]$ we can find a weakly separated pair $(U_i,V_j)$. This is straightforward: if $U_1$ and $V_1$ are not weakly separated then $U_1\subset V_1$ or $V_1\subset U_1$. In the first case $U_1,V_2$ are weakly separated (since $V_1,V_2$ are disjoint), and in the second case $V_1, U_2$ are weakly separated.
		
		Therefore given $\{A_1,A_2,A_3,A_4\}$ and $\{B_1,B_2,B_3,B_4\}$ we can find three weakly separated pairs $(A_i,B_j)$, $(A_k,B_l)$ and $(A_r, B_s)$ with $|\{i,k,r\}|=|\{j,l,s\}|=3$. Given the weakly separated pair $A_i, B_j$ we can pick representative points $x_1$ and $y_1$ from $A_i\setminus B_j$ and $B_j\setminus A_i$ respectively. Similarly, we can pick representative points $x_2,y_2,x_3$ and $y_3$ from the other pairs.
		
		Now we eliminate from $\mathcal{C}$ the sets containing any of the $x_i,y_j$. This removes at most six sets from $\mathcal{C}$ and so some set $C$ remains. If $C\setminus(A_i\cup B_j)$ is non-empty then $A_i, B_j$ and $C$ have the required property. Similarly, we are done if $C\setminus(A_k\cup B_l)$ or $C\setminus(A_r\cup B_s)$ is non-empty. Consider some $c\in C$. The element $c$ belongs to at most one of $A_i, A_k$ or $A_r$ and to at most one of $B_j, B_l$ or $B_s$ so at least one of the pairs $A_i\cup B_j$, $A_k\cup B_l$ or $A_r\cup B_s$ does not contain $c$ and the proof is complete.
	\end{proof}
	
	\begin{proof}[Proof of~\cref{lem:main}]
		We will essentially follow the idea of Lemma~\ref{lem:boosting} as used in the induction step of a proof of the Sauer-Shelah lemma, but we will use~\cref{lem:struct} to obtain structural information about our family in the case where we do not immediately obtain a Venn diagram with six of the inner seven regions filled.
		
		We proceed by induction. The cases $n\le 7$, say, are trivial for $C$ sufficiently large. Now let us assume that $\mathcal{F}$ is a family of size $Cn^2$. For any element $x$ of the ground set, we may consider the families $\mathcal{F}_x=\{F\setminus\{x\}: F\in \mathcal{F}, ~ x\in F\}$ and $\mathcal{F}_{\overline{x}}=\{F:F\in \mathcal{F}, ~ x\not\in F\}$, which may be treated as families on a ground set of size $[n-1]$ by projecting onto $[n]\setminus\{x\}$. Note the similarity with the proof of Lemma~\ref{lem:boosting} -- a difference here is that we will need to consider all the families $\F_x$ and $\F_{\overline{x}}$ rather than a single one.
		
		We let $\mathcal{F}_1(x)=\mathcal{F}_x\cup\mathcal{F}_{\overline{x}}$ and $\mathcal{F}_2(x)=\mathcal{F}_x\cap\mathcal{F}_{\overline{x}}$. We claim that we may assume that $\mathcal{F}_2(x)$ is sufficiently large to apply Lemma~\ref{lem:struct} and find a large, structured subfamily. Indeed if $|\mathcal{F}_1(x)|\ge C(n-1)^2$ then we are done by induction, so we may assume that $|\mathcal{F}_1(x)|<C(n-1)^{2}$ for all $x$. Therefore, for all $x$ we have $|\mathcal{F}_2(x)|>2Cn-C$, since $|\mathcal{F}|=|\mathcal{F}_1(x)|+|\mathcal{F}_2(x)|$.
		
		Now we apply~\cref{lem:struct} to $\mathcal{F}_2(x)$. Observe that if $\mathcal{F}_2(x)$ contains a Venn diagram with five of the inner seven regions filled, then $\mathcal{F}$ contains a Venn diagram with six of the inner seven regions filled by incorporating the element $x$ as in Lemma~\ref{lem:boosting}: for each set $F\in \mathcal{F}_2(x)$, both of the sets $F$ and $F\cup\{x\}$ belong to $\mathcal{F}$. Therefore $\mathcal{F}_2(x)$ contains a pair-like family $\mathcal{H}_x$ of size at least $2Cn-C-\alpha n+\beta> Cn$ for $C$ sufficiently large in terms of $\alpha$. 
		
		We obtain such a family for every $x\in [n]$. We write $\mathcal{H}_x^+$ for the family $\{F\cup\{x\}:F\in \mathcal{H}_x\}\subset\mathcal{F}$ on ground set $[n]$. Observe that
		$$\sum_{x\in [n]}|\mathcal{H}_x^+|>Cn^2.$$
		
		Since the sum of the sizes of the $\mathcal{H}_x^+$ is larger than the size of $\F$, there exists some $F\in \F$ such that $F\in \mathcal{H}_x^+\cap \mathcal{H}_y^+$ for distinct $x$ and $y$. Thus we can pick some $P_1\cup P_2\in \mathcal{H}_x$ so that $F=\{x\}\cup P_1\cup P_2$   (where $P_1$, $P_2$ are parts of the associated basis $\mathcal{P}_x$), and pick some $Q_1\cup Q_2\in \mathcal{H}_y$ so that $F=\{y\}\cup Q_1\cup Q_2$ (where $Q_1, Q_2$ are parts of the associated basis $\mathcal{Q}_y$).
		
		Without loss of generality, we suppose that $y\in P_1$. 
		
		By the definition of pair-like (specifically the existence of the family $\mathcal{P}_{12}$), we can find parts $P_3$ and $P_4$ disjoint from each other and from $F$ such that $P_1\cup P_3\in \mathcal{H}_x$ and $P_1\cup P_4\in \mathcal{H}_x$.
		
		\begin{claim}
			We may assume that $P_1=\{y\}$.
		\end{claim}
		\emph{Proof of claim.} Suppose that there exists $z\neq y$ such that $z\in P_1$. Then we consider the following sets:
		\begin{enumerate}
			\item $A=P_1\cup P_3$ which belongs to $\mathcal{F}$ since $P_1\cup P_3\in \mathcal{H}_x$,
			\item $B=\{x\}\cup P_1\cup P_4$ which belongs to $\mathcal{F}$ since $P_1\cup P_4\in \mathcal{H}_x$, and
			\item $C=Q_1\cup Q_2$ which belongs to $\mathcal{F}$ since $Q_1\cup Q_2\in \mathcal{H}_y$.
		\end{enumerate}
		Then we note that $A$, $B$ and $C$ give a Venn diagram with six out of seven inner regions non-empty: the regions in one set are covered by $P_3,P_4$ and $(Q_1\cup Q_2)\cap P_2$ respectively; the innermost region contains $z$, and points $x$ and $y$ go in $\overline{A}\cap B\cap C$ and $A\cap B \cap \overline{C}$ respectively. 
		
		So we must have that $P_1=\{y\}$.\hfill$\lozenge$
		
		By the same argument, we also have the following.
		\begin{claim}
			We may assume that $Q_1=\{x\}$.\hfill$\lozenge$
		\end{claim}
		
		Therefore $Q_2=P_2=F\setminus\{x,y\}$. 
		
		Let $\mathcal{P}_{21}=\{P\in\mathcal{P}_x:P\cup P_2\in \mathcal{H}_x\}$. Since $\mathcal{H}_x$ is pair-like and $P_1\cup P_2\in \mathcal{H}_x$, there are at least 10 sets $A_1,\dots,A_{10}$ in $\mathcal{P}_{21}$ which are pairwise disjoint and disjoint from $P_1\cup P_2$. Let $\mathcal{A}=\{A_1,\dots,A_{10}\}$.
		
		Similarly, let $\mathcal{Q}_{21}=\{Q\in\mathcal{Q}_y:Q\cup Q_2\in \mathcal{H}_y\}$. Since $\mathcal{H}_y$ is pair-like and $Q_1\cup Q_2\in \mathcal{H}_y$, there are at least 10 sets $B_1,\dots,B_{10}$ in $\mathcal{Q}_{21}$ which are pairwise disjoint and disjoint from $Q_1\cup Q_2$. Let $\mathcal{B}=\{B_1,\dots,B_{10}\}$.
		
		Finally, we let $\mathcal{P}_{12}=\{P\in\mathcal{P}_x:P\cup P_1\in \mathcal{H}_x\}$. Since $\mathcal{H}_x$ is pair-like and $P_1\cup P_2\in \mathcal{H}_x$, there are at least 10 sets $C_1,\dots,C_{10}$ in $\mathcal{P}_{12}$ which are pairwise disjoint and disjoint from $P_1\cup P_2$. Let $\mathcal{C}=\{C_1,\dots,C_{10}\}$.
		
		The families $\mathcal{A},\mathcal{B}$ and $\mathcal{C}$ are three families of disjoint subsets on the same ground set $[n]\setminus F$. Thus we may apply~\cref{lem:useful} to find three sets $A\in\mathcal{A}, B\in\mathcal{B}, C\in\mathcal{C}$ such that $A\setminus(B\cup C), B\setminus(A\cup C)$ and $C\setminus (A\cup B)$ are each non-empty.
		
		%
		
		Now consider the three sets
		\begin{enumerate}
			\item $A'=\{x\}\cup P_2\cup A$ which belongs to $\mathcal{F}$ since $P_2\cup A\in \mathcal{H}_x$,
			\item $B'=\{y\}\cup Q_2\cup B$ which belongs to $\mathcal{F}$ since $Q_2\cup B\in \mathcal{H}_y$, and
			\item $C'=\{x\}\cup P_1\cup C$ which belongs to $\mathcal{F}$ since $P_1\cup C\in \mathcal{H}_x$.
		\end{enumerate}
		
		We see that the point $x$ belongs to $A'\cap\overline{B'}\cap C'$, the point $y$ belongs to $\overline{A'}\cap B'\cap C'$, the set $P_2=Q_2$ belongs to $A'\cap B'\cap\overline{C'}$ and the sets $A, B$ and $C$ deal with the regions in a single set. This gives six out of seven inner regions non-empty and we are done.
	\end{proof}
	
	\begin{corollary}\label{cor:almostfinal}
		There exists an absolute constant $C$ such that any family $\F$ on ground set $[n]$ of size $|\F|\geq Cn^3$ contains a 3-Venn with all seven inner regions filled.
	\end{corollary}
	\begin{proof}
	We combine \cref{lem:main} with \cref{lem:boosting} using the partial 3-Venn diagram $\mathrm{V}'(A,B,C)=\mathrm{V}_3(A,B,C)\setminus\{\overline{A}\cap\overline{B}\cap\overline{C}\}$.
	\end{proof}
	
	\begin{proof}[Proof of~\cref{thm:main}]
		Let $\F$ be a family on ground set $[n]$ with $|\F|\geq (2C+3D)n^3$, where $C$ is the constant given by~\cref{cor:almostfinal} and $D$ is the constant given by~\cref{cor:cor1}. Partition $\F$ into $\F= \F_{\text{small}}~\dot{\cup}~\F_{\text{mid}}~\dot{\cup}~\F_{\text{big}} $ where 
		$$\F_{\text{small}}=\{F\in\F:|F|< n/3\},\quad\F_{\text{mid}}=\{F\in\F:n/3\leq |F|\leq 2n/3\},$$
		$$\F_{\text{big}}=\{F\in\F:|F|> 2n/3\}. $$
		If $|\F_{\text{mid}}|\geq 3Dn^3$ then by~\cref{cor:cor1} it contains a full 3-Venn diagram. Hence we may assume $|\F_{\text{small}} \cup \F_{\text{big}}|\geq 2Cn^3$. Without loss of generality we may assume $|\F_{\text{small}}|\geq Cn^3$ -- otherwise we have $|\F_{\text{big}}|\geq Cn^3$ and we may simply replace all sets by their complements. By~\cref{cor:almostfinal}, $\F_{\text{small}}$ contains three sets $A,B,C$ with all seven inner regions of their Venn diagram being non-empty. However, as $|A|,|B|,|C|<n/3$ we also have $\overline{A}\cap \overline{B}\cap \overline{C}\neq \emptyset$ and the proof is complete.
	\end{proof}
	
	\section{Conclusion and open questions}
	
	Denote by $f_k(n)$ the maximum size of a family $\F$ on ground set $[n]$ with $\DVC(\F)\leq k$. Then $f_1(n)=4n-2$ and the main result of this paper is that $f_2(n)=\Theta(n^3)$. A natural question is as follows.
	\begin{question}\label{ques:kvenn}
		Fix some $k\geq 3$. What is the order of magnitude of $f_k(n)$?
	\end{question}
	
	Answering Question~\ref{ques:kvenn} for all $k$ would complete a dual version of the Sauer-Shelah lemma.
	
	The trivial bounds are
	$$cn^{2^{k-1}-1}\leq f_k(n)\leq Cn^{2^k-1}.$$
	Indeed, the lower bound is the straightforward construction of taking all sets of size $2^{k-1}-1$. If this did contain a Venn diagram on $k$ sets $A_1,\ldots,A_k$ with all $2^k$ regions non-empty, then each $A_i$ would be partitioned into $2^{k-1}$ non-empty regions which is impossible as $|A_i|<2^{k-1}$. The upper bound follows from the observation that if $|\F|\geq Cn^{2^k-1}$ then $\mathrm{VC}(\F)\geq 2^k$. So by~\cref{lem:sauershelah} $\F$ shatters a set of size $2^k$ and hence it contains a $k$-Venn diagram. We believe that the lower bound gives the right order of magnitude.

	\bibliographystyle{plain}
	\bibliography{vennbib}

\end{document}